\RequirePackage{fix-cm}
\documentclass[smallextended]{article}       % onecolumn (second format)
\usepackage{color}
\usepackage{amsfonts}
\usepackage{amsmath}
\usepackage{graphicx}
\usepackage{wrapfig}
\usepackage{amscd}
\usepackage{amssymb}
\usepackage{amstext}
\usepackage{mathrsfs}
\usepackage{amsthm}

 \newtheorem{theorem}{Theorem}[section]
 \newtheorem{proposition}{Proposition}[section]
 
 \newtheorem{corollary}{Corollary}[section]

 \newtheorem{definition}{Definition}[section]
 \newtheorem{remark}{Remark}[section] 
 \numberwithin{equation}{section}
\begin{document}

\title{Beta operators with Jacobi weights%\thanks{Grants or other notes
}

%%%% AUTOR %%%%
\date{Heiner Gonska, Ioan Ra\c sa and Elena Dorina St\u anil\u a  }


\maketitle

\begin{abstract}
We discuss Beta operators with Jacobi weights on $C[0,1]$ for $\alpha,\beta\geq-1$, thus including the discussion of three limiting cases. Emphasis is on the moments and their asymptotic behavior. Extended Voronovskaya-type results and a discussion concerning the over-iteration of the operators are included.
\end{abstract}

\section{Introduction }
Many operators arising in the theory of positive linear operators are compositions of other mappings of this type. Many times the classical Bernstein operator $B_n$ given for $f\in C[0,1], n\in\mathbb{N}$ and $x\in [0,1]$ by
\begin{equation}\label{def_bernstein}
B_n(f;x):=\sum\limits_{k=0}^n{n\choose k}x^k(1-x)^{n-k}f\left(\frac{k}{n}\right),\;0\leq k\leq n,
\end{equation}
is one of the building blocks. Other frequently used factor operators are Beta-type operators $\mathcal{B}_r^{\alpha,\beta}$ of various kinds which will be further discussed in this note.\\
The best known examples are the genuine Bernstein-Durrmeyer operators $U_n$, the original Bernstein-Durrmeyer operators $M_n$, their analogies $M_n^{\alpha,\beta}$ with Jacobi weights, certain Stancu operators $S_n^{\alpha}$, to name just a few. A complete list will be given in the third author's forthcoming thesis on Bernstein-Euler-Jacobi (BEJ) operators.\\
Here we focus on the building blocks $\mathcal{B}_r^{\alpha,\beta}$ for natural values of $r$ and $\alpha,\beta\geq -1$, and on their moments of all orders. As is well known, knowledge of their behavior is essential
for asymptotic statements as, for example, Voronovskaya-type results. We conclude this paper with a discussion concerning over-iterated operators $\mathcal{B}_n^{\alpha,\beta}$.
%===========================================
\section{Definition of operators $\mathcal{B}_n^{\alpha,\beta}$}
%===========================================

\begin{definition}\label{def_beta_op}
For $f\in C[0,1]$, and $x\in[0,1]$  we define\\
\begin{itemize}
\item[(i)] in case $\alpha=\beta=-1$:
$$\mathcal{B}_n^{-1,-1}(f;x)=
\begin{cases}
f(0), x=0;\\
\dfrac{\int\limits_0^1t^{nx-1}(1-t)^{n-nx-1}f(t)dt}{B(nx,n-nx)}, 0<x<1;\\
f(1), x=1.
\end{cases}$$
\item[(ii)] in case $\alpha=-1, \beta>-1$:
$$\mathcal{B}_n^{-1,\beta}(f;x)=
\begin{cases}
f(0), x=0;\\
\dfrac{\int\limits_0^1t^{nx-1}(1-t)^{n-nx+\beta}f(t)dt}{B(nx,n-nx+\beta+1)}, 0<x\leq1.
\end{cases}$$
\item[(iii)] in case $\alpha>-1,\beta=-1$:
$$\mathcal{B}_n^{\alpha,-1}(f;x)=
\begin{cases}
\dfrac{\int\limits_0^1t^{nx+\alpha}(1-t)^{n-nx-1}f(t)dt}{B(nx+\alpha+1,n-nx)}, 0\leq x<1;\\
f(1), x=1.
\end{cases}$$
\item[(iv)] in case $\alpha,\beta>-1$:
$$\mathcal{B}_n^{\alpha,\beta}(f;x)=\dfrac{\int\limits_0^1t^{nx+\alpha}(1-t)^{n-nx+\beta}f(t)dt}{B(nx+\alpha+1,n-nx+\beta+1)}, 0\leq x\leq 1.$$
\end{itemize}
\end{definition}
%-------------------------------------------------------------
\begin{remark} When discussing this class of operators one must refer to the papers of M\"uhlbach \cite{Muhlbach:1972} and Lupa\c s in \cite{Lupas:1972} where the first special cases were considered.\\
Case $\alpha=\beta=-1$: This case can be traced back to a paper by M\"uhlbach \cite{Muhlbach:1972} who used a real number $\frac{1}{\lambda}>0$ instead of the natural $n$ in the definition above. The same case was investigated by Lupa\c s in \cite{Lupas:1972}, where the operator was denoted by $\overline{\mathbb{B}}_n$ (see \cite[p.63]{Lupas:1972}).\\
Case $\alpha=\beta=0$: These were called Beta operators by Lupa\c s (see \cite[p.37]{Lupas:1972}) and denoted by $\mathbb{B}_n$.
\end{remark}
%-------------------------------------------------------------
%===========================================
\section{Moments and their recursion}
%===========================================
\begin{definition}
Let $\alpha,\beta\geq -1,n> 1, m\in \mathbb{N}_0$ and $x\in[0,1]$, then the moment of order $m$ is defined by
\begin{equation*}
T_{n,m}^{\alpha,\beta}(x)=\mathcal{B}_n^{\alpha,\beta}((e_{1}-xe_{0})^m;x).
\end{equation*}
\end{definition}
%---------
\begin{theorem}\label{beta_th1.2}
\begin{equation}\label{beta_eq1.0}
T_{n,0}^{\alpha,\beta}(x)=1, \;T_{n,1}^{\alpha,\beta}(x)=\dfrac{\alpha+1-(\alpha+\beta+2)x}{n+\alpha+\beta+2}
\end{equation}
and for $m\geq 1$ we have the following recursion formula
\begin{eqnarray}\label{beta_eq1.1}
(n+m+\alpha+\beta+2)T_{n,m+1}^{\alpha,\beta}(x)=mXT_{n,m-1}^{\alpha,\beta}(x)+\\+[m+\alpha+1-(2m+\alpha+\beta+2)x]T_{n,m}^{\alpha,\beta}(x)\nonumber
\end{eqnarray}
where $X=x(1-x)$.
\end{theorem}
\begin{proof}
Below we will repeatedly use the function $\psi(t)=t(1-t), t\in
[0,1]$. Let $f\in C^1[0,1], \alpha,\beta\geq -1, 0<x<1$. Then
$$
\mathcal{B}_n^{\alpha,\beta}(\psi f';x)=
\dfrac{\int\limits_0^1t^{nx+\alpha}(1-t)^{n-nx+\beta}t(1-t)f'(t)dt}{B(nx+\alpha+1,n-nx+\beta+1)}.
$$
Using integration by parts we obtain

\begin{eqnarray*}
\mathcal{B}_n^{\alpha,\beta}(\psi f';x)=
\dfrac{1}{B(nx+\alpha+1,n-nx+\beta+1)}[t^{nx+\alpha+1}(1-t)^{n-nx+\beta+1}f(t)\bigg{|}_0^1\\
-\int\limits_0^1f(t)[(nx+\alpha+1)t^{nx+\alpha}(1-t)^{n-nx+\beta+1}-\\
\;\;\;\;\;\;\;\;\;\;\;\
-(n-nx+\beta+1)t^{nx+\alpha+1}(1-t)^{n-nx+\beta}]dt \\
=\dfrac{\int\limits_0^1f(t)t^{nx+\alpha}(1-t)^{n-nx+\beta}[t(n-nx+\beta+1)-(1-t)(nx+\alpha+1)]dt}{B(nx+\alpha+1,n-nx+\beta+1)}\\
=\dfrac{\int\limits_0^1f(t)t^{nx+\alpha}(1-t)^{n-nx+\beta}[n(t-x)-(\alpha+1)+t(\alpha+\beta+2)]dt}{B(nx+\alpha+1,n-nx+\beta+1)}
\end{eqnarray*}
and taking into consideration the identity
$$
\begin{array}{l}
n(t-x)-(\alpha+1)+t(\alpha+\beta+2)
=\\=\left((e_1-xe_0)(n+\alpha+\beta+2)+[x(\alpha+\beta+2)-(\alpha+1)]e_0\right)(t)
\end{array}
$$
we can now write
\begin{equation}\label{beta_eq1.2}
\mathcal{B}_n^{\alpha,\beta}(\psi f';x) =
\mathcal{B}_n^{\alpha,\beta}
([(e_1-xe_0)(n+\alpha+\beta+2)+(x(\alpha+\beta+2)-(\alpha+1))e_0]f;x).
\end{equation}
In (\ref{beta_eq1.2}) we choose $f=(e_1-xe_0)^m$ and use the fact that
$t(1-t)=(X+X'(e_1-xe_0)-(e_1-xe_0)^2)(t)$:
$$
\begin{array}{l}
m\mathcal{B}_n^{\alpha,\beta}([X(e_{1}-xe_{0})^{m-1}+X'(e_{1}-xe_{0})^{m}-(e_{1}-xe_{0})^{m+1}];x)=\\ \mathcal{B}_n^{\alpha,\beta}([(n+\alpha+\beta+2)(e_{1}-xe_{0})^{m+1}-(\alpha+1-(\alpha+\beta+2)x)(e_{1}-xe_{0})^{m}];x).
\end{array}
$$
The equality above becomes successively:

\begin{eqnarray*}
mXT_{n,m-1}^{\alpha,\beta}(x)+mX'T_{n,m}^{\alpha,\beta}(x)-mT_{n,m+1}^{\alpha,\beta}(x)=
(n+\alpha+\beta+2)T_{n,m+1}^{\alpha,\beta}(x)-\\-[\alpha+1-(\alpha+\beta+2)x]T_{n,m}^{\alpha,\beta}(x);
\end{eqnarray*}
\begin{eqnarray*}
%(m+n+\alpha+\beta+2)T_{n,m+1}^{\alpha,\beta}(x)=mXT_{n,m-1}^{\alpha,\beta}(x)+[m(1-2x)+\alpha+1-(\alpha+\beta+2)x]T_{n,m}^{\alpha,\beta}(x)\\
(m+n+\alpha+\beta+2)T_{n,m+1}^{\alpha,\beta}(x)=mXT_{n,m-1}^{\alpha,\beta}(x)+\\ +[m+\alpha+1-(\alpha+\beta+2+2m)x]T_{n,m}^{\alpha,\beta}(x).
\end{eqnarray*}
So (\ref{beta_eq1.1}) is established for $0<x<1$. Due to the continuity, it is valid also for $x\in\{0,1\}$.
\end{proof}
%-------------------------------------------------------------
%Here are some important particular cases of the Beta operator which were introduced by Lupa\c s in \cite{Lupas}.
In particular we have:
\begin{corollary}
For $\alpha=\beta=0$ we have $\mathcal{B}_n^{0,0}=\mathbf{\mathbb{B}}_{n}$ (Lupa\c s notation) with the corresponding recurrence formula for the moments:
$$
\begin{array}{l}
(n+m+2)T_{n,m+1}^{0,0}(x)=mXT_{n,m-1}^{0,0}(x)+(m+1)X'T_{n,m}^{0,0}(x)
\end{array}
$$
where $T_{n,0}^{0,0}(x)=1, \;T_{n,1}^{0,0}(x)=\dfrac{X'}{n+2}$.\\
For $\alpha=\beta=-1$ we have $\mathcal{B}_n^{-1,-1}=\mathbf{\overline{\mathbb{B}}}_{n}$ (Lupa\c s notation). Then the recurrence formula becomes
$$
\begin{array}{l}
(n+m)T_{n,m+1}^{-1,-1}(x)=mXT_{n,m-1}^{-1,-1}(x)+mX'T_{n,m}^{-1,-1}(x)
\end{array}
$$
where $T_{n,0}^{-1,-1}(x)=1, \;T_{n,1}^{-1,-1}(x)=0$.
\end{corollary}

%-------------------------------------------------------------
In the sequel we denote by $(a)^{\overline{r}}=a(a+1)\cdot...\cdot(a+r-1)$ the rising factorial function.
The next proposition contains another kind of recurrence formula for the moments.
\begin{proposition}
Let $i\geq 0$ and $j\geq 0$ be integers. Then
\begin{equation}\label{beta_eq1.4}
T_{n,m}^{\alpha+i,\beta+j}(x)=\dfrac{(n+\alpha+\beta+2)^{\overline{i+j}}}{(nx+\alpha+1)^{\overline{i}}(nx+\beta+1)^{\overline{j}}}\sum\limits_{k=0}^{i+j}\dfrac{[x^{i}(1-x)^j]^{(k)}}{k!}T_{n,m+k}^{\alpha,\beta}(x).
\end{equation}
\end{proposition}
\begin{proof}
Using the definition of the Beta operator it is easy to show that
\begin{equation}\label{beta_eq1.5}
\mathcal{B}_n^{\alpha,\beta}(t^{i}(1-t)^jf(t);x)=\dfrac{(nx+\alpha+1)^{\overline{i}}(nx+\beta+1)^{\overline{j}}}{(n+\alpha+\beta+2)^{\overline{i+j}}}\mathcal{B}_n^{\alpha+i,\beta+j}(f(t);x).
\end{equation}
The following equation
\begin{equation}\label{beta_eq1.3}
t^{i}(1-t)^j=\sum\limits_{k=0}^{i+j}\dfrac{[x^{i}(1-x)^j]^{(k)}}{k!}(t-x)^k
\end{equation}
is a consequence of Taylor's formula.
%\begin{proof}
%We denote $h(t):=t^{i}(i-t)^j$ and using Tylor's formula for polynomials we get
%$$
%h(t)=\sum\limits_{k=0}^{i+j}\dfrac{h^{(k)}(x)}{k!}(t-x)^k
%$$
%We then apply Leibniz's formula for the $k$-th derivative of a product
%$$
%\begin{array}{rcl}
%h(t)&=&\sum\limits_{k=0}^{i+j}\dfrac{(t-x)^k}{k!}\sum\limits_{l=0}^k{k \choose l}(x^{i})^{(l)}
%[(1-x)^j]^{(k-l)}\\
%&=&\sum\limits_{k=0}^{i+j}\dfrac{(t-x)^k}{k!}\sum\limits_{l=0}^k{k \choose l}i\cdot...\cdot(i-l+1)
%j\cdot...\cdot(j-k+l+1)x^{i-l}(-1)^{k-l}(1-x)^{j-k+l}\\
%&=&\sum\limits_{k=0}^{i+j}\dfrac{(t-x)^k}{k!}\sum\limits_{l=0}^k\dfrac{k!}{l!(k-l)!}\cdot \dfrac{i!}{(i-l)!}\cdot \dfrac{j!}{(j-k+l)!}(-1)^{k-l}x^{i-l}(1-x)^{j-k+l}\\
%&=&\sum\limits_{k=0}^{i+j}\dfrac{(t-x)^k}{k!}\sum\limits_{l=0}^k{i\choose l}{j\choose k-l}k!
%(-1)^{k-l}x^{i-l}(1-x)^{j-k+l}\\
%&=&\sum\limits_{k=0}^{i+j}\dfrac{[x^{i}(1-x)^j]^{(k)}}{k!}(t-x)^k.
%\end{array}
%$$
%\end{proof}
Next using (\ref{beta_eq1.3}) and the fact that the Beta operator is linear we get
\begin{equation}\label{beta_eq1.6}
\mathcal{B}_n^{\alpha,\beta}(t^{i}(1-t)^jf(t);x)=\sum\limits_{k=0}^{i+j}\dfrac{[x^{i}(1-x)^j]^{(k)}}{k!}\mathcal{B}_n^{\alpha,\beta}((t-x)^kf(t);x).
\end{equation}
Combining (\ref{beta_eq1.5}) and (\ref{beta_eq1.6}) we arrive at
\begin{equation*}
\begin{array}{rcl}
\mathcal{B}_n^{\alpha+i,\beta+j}(f(t);x)&=&\dfrac{(n+\alpha+\beta+2)^{\overline{i+j}}}{(nx+\alpha+1)^{\overline{i}}(nx+\beta+1)^{\overline{j}}}\times\\&&\times\sum\limits_{k=0}^{i+j}\dfrac{[x^{i}(1-x)^j]^{(k)}}{k!}\mathcal{B}_n^{\alpha,\beta}((t-x)^kf(t);x)
\end{array}
\end{equation*}
For $f(t)=(t-x)^m$  we obtain (\ref{beta_eq1.4}).
\end{proof}
\begin{remark}
Another recurrence formula for the moments of $\mathcal{B}_n^{-1,-1}$ can be found in \cite[Satz 3]{Muhlbach:1972}.
\end{remark}
%===========================================
\section{The moments of order two}
%===========================================
Since the second moment controls to a certain extent the approximation properties of
$\mathcal{B}_n^{\alpha,\beta}$, it is useful to have a closer look at it.
From Theorem \ref{beta_th1.2} we obtain
\begin{eqnarray}\label{beta_new.eq2.1}
T_{n,2}^{\alpha,\beta}(x)=\dfrac{(\alpha+1)(\alpha+2)+(n-2(\alpha+1)(\alpha+\beta+3))x}{(n+\alpha+\beta+2)(n+\alpha+\beta+3)}+\\+\dfrac{(-n+6+(\alpha+\beta)(\alpha+\beta+5))x^2}{(n+\alpha+\beta+2)(n+\alpha+\beta+3)}.\nonumber
\end{eqnarray}
(I). First, let us remark that
\begin{equation}\label{beta_new.eq2.2}
\lim\limits_{\alpha\rightarrow\infty}T_{n,2}^{\alpha,\beta}(x)=(1-x)^2, \;\mbox{uniformly on}\; [0,1],
\end{equation}
and
\begin{equation}\label{beta_new.eq2.3}
\lim\limits_{\beta\rightarrow\infty}T_{n,2}^{\alpha,\beta}(x)=x^2, \;\mbox{uniformly on}\; [0,1].
\end{equation}
Roughly speaking, a large value of $\alpha$ (with a fixed $\beta$) suggests a better approximation near $1$, and we draw a similar conclusion from (\ref{beta_new.eq2.3}).\\
(II). Now let $\beta=\alpha\geq-1$. Consider the sequence $s_n:=\dfrac{\sqrt{4n+1}-5}{4}, n\geq 1$. In this case,
\begin{equation}\label{beta_new.eq2.4}
T_{n,2}^{\alpha,\alpha}(x)=\dfrac{(\alpha+1)(\alpha+2)-(-n+6+2\alpha(2\alpha+5))x(1-x)}{(n+2\alpha+2)(n+2\alpha+3)}.
\end{equation}
Therefore,
\begin{equation}\label{beta_new.eq2.5}
T_{n,2}^{\alpha,\alpha}(0)=T_{n,2}^{\alpha,\alpha}(1)=\dfrac{(\alpha+1)(\alpha+2)}{(n+2\alpha+2)(n+2\alpha+3)},
\end{equation}
and
\begin{equation}\label{beta_new.eq2.6}
T_{n,2}^{\alpha,\alpha}\left(\frac{1}{2}\right)=\dfrac{1}{4(n+2\alpha+3)}.
\end{equation}
\begin{itemize}
\item[(i)] If $-1\leq \alpha<s_n$, the graph of $T_{n,2}^{\alpha,\alpha}$ has the following form:\\
\setlength{\unitlength}{1cm}
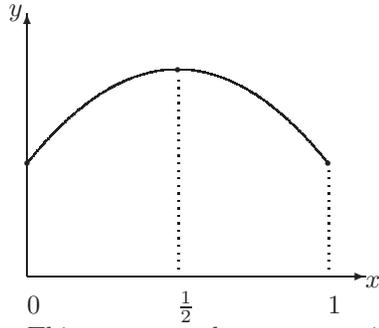
\begin{picture}(5,4)(0,0)
\put(0,0){\vector(1,0){4.5}}
\put(4.5,-.15){$x$}
\put(0,-.5){$0$}
\put(2,-.5){$\frac{1}{2}$}
\put(4,-.5){$1$}
\put(0,0){\vector(0,1){3.5}}
\put(-.15,3.5){\makebox(0,0){$y$}}
\qbezier(0.0,1.5)(2,4)(4,1.5)
\multiput(2,0)(0,0.1){27}
{\line(1,0){0.031}}
\multiput(4,0)(0,0.1){15}
{\line(1,0){0.031}}
\put(2,2.75){\circle*{0.07}}
\put(0,1.5){\circle*{0.07}}
\put(4,1.5){\circle*{0.07}}
\linethickness{.075mm}
\end{picture}\vspace{0.5cm}\\
This suggests a better approximation near the end points.
\item[(ii)] If $\alpha=s_n, T_{n,2}^{\alpha,\alpha}$ is a constant function, namely
\begin{equation}\label{beta_new.eq2.7}
T_{n,2}^{s_n,s_n}(x)=\left(\dfrac{\sqrt{4n+1}-1}{4n}\right)^2, x\in[0,1].
\end{equation}
\item[(iii)] For $\alpha>s_n$, the graph looks like\\
\setlength{\unitlength}{1cm}
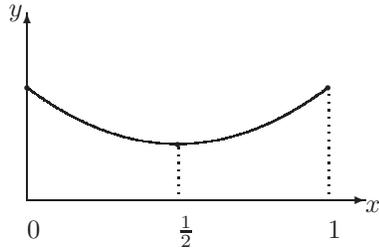
\begin{picture}(5,4)(0,0)
\put(0,0){\vector(1,0){4.5}}
\put(4.5,-.15){$x$}
\put(0,-.5){$0$}
\put(2,-.5){$\frac{1}{2}$}
\put(4,-.5){$1$}
\put(0,0){\vector(0,1){2.5}}
\put(-.15,2.5){\makebox(0,0){$y$}}
\qbezier(0.0,1.5)(2,0)(4,1.5)
\multiput(2,0)(0,0.1){8}
{\line(1,0){0.031}}
\multiput(4,0)(0,0.1){15}
{\line(1,0){0.031}}
\put(2,0.75){\circle*{0.07}}
\put(0,1.5){\circle*{0.07}}
\put(4,1.5){\circle*{0.07}}
\linethickness{.075mm}
\end{picture}\vspace{0.5cm}\\
and indicates a better approximation near $\frac{1}{2}$.
\item[(iv)] In the extreme cases, when $\alpha=-1$, respectively $\alpha\rightarrow\infty$, we have
$T_{n,2}^{-1,-1}(x)=\dfrac{x(1-x)}{n+1}$, respectively $\lim\limits_{\alpha\rightarrow\infty}T_{n,2}^{\alpha,\alpha}(x)=
\left(\dfrac{1-2x}{2}\right)^2$.
\end{itemize}
%===========================================
\section{Asymptotic formulae}
%===========================================
Here we present first two asymptotic formulae for higher order moments of $\mathcal{B}_n^{\alpha,\beta}$ in order to arrive at Voronovskaya-type results.
\begin{theorem}\label{beta_lm3.1}
For $\alpha,\beta\geq -1$ and all $l\geq 1$ one has
\begin{eqnarray}\label{beta_eq2.1}
(P_l):\;
\begin{cases}
\lim\limits_{n\rightarrow\infty}n^lT_{n,2l}^{\alpha,\beta}(x)=(2l-1)!!X^l,\\
\lim\limits_{n\rightarrow\infty}n^lT_{n,2l-1}^{\alpha,\beta}(x)=
X^{l-1}\left[(l-1)!2^{l-1}X'\sum\limits_{k=1}^{l-1}\dfrac{(2k-1)!!}{(2k-2)!!}+\right.\\\left. +(2l-1)!!(\alpha+1-(\alpha+\beta+2)x)\right].
\end{cases}
\end{eqnarray}
The convergence is uniform on $[0,1]$.
\end{theorem}
\begin{proof}
We shall prove the proposition by induction on $l\geq 1$.
$T_{n,1}^{\alpha,\beta}$ and $T_{n,2}^{\alpha,\beta}$ are given by (\ref{beta_eq1.0}), respectively (\ref{beta_new.eq2.1}), and it is easy to prove that $(P_1)$ is true. Suppose that $(P_l)$ is true. According to ( \ref{beta_eq1.1}) and (\ref{beta_eq2.1}),
$$
\begin{array}{l}
\lim\limits_{n\rightarrow\infty}n^{l+1}T_{n,2l+1}^{\alpha,\beta}(x)=
\lim\limits_{n\rightarrow\infty}n^{l+1}\dfrac{2lX}{n+2l+\alpha+\beta+2}T_{n,2l-1}^{\alpha,\beta}(x)+\\+
\lim\limits_{n\rightarrow\infty}n^{l+1}\dfrac{2l+\alpha+1-(4l+\alpha+\beta+2)x}{n+2l+\alpha+\beta+2}T_{n,2l}^{\alpha,\beta}(x)\\
= 2lX^l\left[(l-1)!2^{l-1}X'\sum\limits_{k=1}^{l-1}\dfrac{(2k-1)!!}{(2k-2)!!}+(2l-1)!!(\alpha+1-(\alpha+\beta+2)x)\right]+\\
 +[2l+\alpha+1-(4l+\alpha+\beta+2)x](2l-1)!!X^l\\
=X^l[2^ll!X'\sum\limits_{k=1}^{l-1}\dfrac{(2k-1)!!}{(2k-2)!!}\\ +(2l-1)!!(2l(\alpha+1)-2l(\alpha+\beta+2)x+2l+\alpha+1-(4l+\alpha+\beta+2)x)]\\
=X^l[2^ll!X'\sum\limits_{k=1}^{l}\dfrac{(2k-1)!!}{(2k-2)!!}-(2l)!!X'\dfrac{(2l-1)!!}{(2l-2)!!}+\\ +
(2l-1)!!((2l+1)(\alpha+1-(\alpha+\beta+2)x)+2l-4lx]\\
=X^l\left[2^ll!X'\sum\limits_{k=1}^{l}\dfrac{(2k-1)!!}{(2k-2)!!}+(2l+1)!!(\alpha+1-(\alpha+\beta+2)x\right]
\end{array}
$$
and this proves the first formula in (\ref{beta_eq2.1}) for $l+1$ instead of $l$.
Similarly,
$$
\begin{array}{l}
\lim\limits_{n\rightarrow\infty}n^{l+1}T_{n,2l+2}^{\alpha,\beta}(x)=
\lim\limits_{n\rightarrow\infty}n^{l+1}\dfrac{(2l+1)X}{n+2l+1+\alpha+\beta+2}T_{n,2l}^{\alpha,\beta}(x)+\\ +
\lim\limits_{n\rightarrow\infty}n^{l+1}\dfrac{2l+1+\alpha+1-(4l+2+\alpha+\beta+2)x}{n+2l+1+\alpha+\beta+2}T_{n,2l+1}^{\alpha,\beta}(x)\\
=(2l+1)X(2l-1)!!X^l=(2l+1)!!X^{l+1},
\end{array}
$$
which is the second formula in  (\ref{beta_eq2.1}) for $l+1$ instead of $l$. This concludes the proof by induction.
\end{proof}
%-------------------------------------------------------------
The following result of Sikkema (see \cite[p. 241]{Sikkema:1975}) will be used below.
Note also the 1962 result of Mamedov \cite{Mamedov:1962} dealing with a similar problem.
\begin{theorem}\label{beta_sikkema_th}
Let $L_n:B[a,b]\rightarrow C[c,d], [c,d]\subseteq [a,b]$, be a sequence of positive linear operators. Let the function $f\in B[a,b]$ be $q-$times differentiable at $x\in [c,d]$, where $q\geq 2$ is a natural number. Let $\varphi:\mathbb{N}\rightarrow\mathbb{R}$ be a function such that
\begin{itemize}
\item[(i)] $\lim\limits_{n\rightarrow\infty} \varphi(n)=\infty$,
\item[(ii)] $L_n((e_1-x)^q;x)=\dfrac{c_q(x)}{\varphi(n)}+o\left(\frac{1}{\varphi(n)}\right), \; n\rightarrow\infty$, where $c_q(x)$ does not depend on $n$,
\item[(iii)] there exists an even number $m> q$ such that
$L_n((e_1-x)^m;x)=o\left(\frac{1}{\varphi(n)}\right), \; $ $ n\rightarrow\infty$.
\end{itemize}
Then 
$$
\lim\limits_{n\rightarrow\infty}\varphi(n)\left\{L_n(f;x)-\sum\limits_{r=0}^q\dfrac{L_n((e_1-x)^r;x)}{r!}f^{(r)}(x)\right\}=0.
$$
\end{theorem}
\begin{corollary}\label{beta_sikkema_cor}
\begin{itemize}
\item[(i)] 
Theorem \ref{beta_sikkema_th} can be rewritten in the form
$$
\lim\limits_{n\rightarrow\infty}\varphi(n)\left\{L_n(f;x)-\sum\limits_{r=0}^{q-1}\dfrac{L_n((e_1-x)^r;x)}{r!}f^{(r)}(x)\right\}=c_q(x)\dfrac{f^{(q)}(x)}{q!}.
$$
\item[(ii)] If in addition to the assumption of Theorem \ref{beta_sikkema_th}, one assumes that
$$
L_n((e_1-x)^r;x)=\dfrac{c_r(x)}{\varphi(n)}+o\left(\frac{1}{\varphi(n)}\right), \; n\rightarrow\infty, r=1,2,...,q,
$$
where the functions $c_r$ are independent of $n$, then one also has
$$
\lim\limits_{n\rightarrow\infty}\varphi(n)\left\{L_n(f;x)-f(x)L_n(e_0;x)\right\}=\sum\limits_{r=1}^{q}c_r(x)\dfrac{f^{(r)}(x)}{q!}.
$$
That is, all derivatives now appear on the right hand side which is independent of $n$.
\end{itemize}
\end{corollary}
%-------------------------------------------------------------
As a consequence of Corollary \ref{beta_sikkema_cor} (ii) we have the following Voronovskaya-type relation.
\begin{corollary}
Let $f\in C^2[0,1]$. Then
\begin{equation}
\lim\limits_{n\rightarrow\infty}n\left\{\mathcal{B}_n^{\alpha,\beta}(f;x)-f(x)\right\}=
\dfrac{x(1-x)}{2}f''(x)+[\alpha+1-(\alpha+\beta+2)x]f'(x),
\end{equation}
uniformly on $[0,1]$.
\end{corollary}
\begin{proof}
For $\varphi(n)=n$ and $q=2$ as given in Corollary \ref{beta_sikkema_cor} (ii),
$$
\lim\limits_{n\rightarrow\infty}n\left\{\mathcal{B}_n^{\alpha,\beta}(f;x)-f(x)\right\}=\sum\limits_{r=1}^2c_{r}(x)\dfrac{f^{(r)}(x)}{r!}=c_{1}(x)\dfrac{f'(x)}{1!}+c_{2}(x)\dfrac{f''(x)}{2!}
$$
where
$c_{r}(x)=\lim\limits_{n\rightarrow\infty}nT_{n,r}^{\alpha,\beta}(x).$
By using Lemma \ref{beta_lm3.1} with $l=1$ we get
$$
\begin{array}{rcl}
c_{1}(x)&=&\alpha+1-(\alpha+\beta+2)x\\
c_{2}(x)&=&X,
\end{array}
$$
and this concludes the proof.
\end{proof}
%-------------------------------------------------------------
%-------------------------------------------------------------
\begin{remark}
As a consequence of Lemma \ref{beta_lm3.1} and Corollary \ref{beta_sikkema_cor} (i) we deduce similarly that for $f\in C^{2l}[0,1]$,
\begin{equation}
\lim\limits_{n\rightarrow\infty}n^l\left\{\mathcal{B}_n^{\alpha,\beta}(f(t);x)-\sum\limits_{k=0}^{2l-1}\dfrac{f^{(k)}(x)}{k!}T_{n,k}^{\alpha,\beta}(x)\right\}=\dfrac{(2l-1)!!}{(2l)!}X^lf^{(2l)}(x), l\geq 1.
\end{equation}
From this we get also
\begin{eqnarray}\label{beta_eq1.48}
\lim\limits_{n\rightarrow\infty}n^l\left\{\mathcal{B}_n^{\alpha,\beta}(f(t);x)-\sum\limits_{k=0}^{2l-2}\dfrac{f^{(k)}(x)}{k!}T_{n,k}^{\alpha,\beta}(x)\right\}=\dfrac{(2l-1)!!}{(2l)!}X^lf^{(2l)}(x)+\nonumber\\+\dfrac{X^{l-1}}{(2l-1)!}\left[(l-1)!2^{l-1}X'\sum\limits_{k=1}^{l-1}\dfrac{(2k-1)!!}{(2k-2)!!}+\right. \;\;\;\;\;\;\;\;\;\;\;\;\;\;\;\;\;\;\;\;\;\;\;\;\;\;\;\;\;\;\;\;\;\;\;\;\;\;\\+(2l-1)!!(\alpha+1-(\alpha+\beta+2)x)\bigg]f^{(2l-1)}(x). \;\;\;\;\;\;\;\;\;\;\;\;\;\;\;\;\;\;\;\;\;\;\;\;\;\;\;\;\nonumber
\end{eqnarray}
\end{remark}
\begin{remark}
In order to compare the above with a special previous result for the case $\alpha=\beta=-1$ we manipulate the left hand side of (\ref{beta_eq1.48}) for $l=2$ by writing
$$
\begin{array}{l}
\lim\limits_{n\rightarrow\infty}n[n(\mathcal{B}_n^{\alpha,\beta}(f(t);x)-f(x))-(\alpha+1-(\alpha+\beta+2)x)f'(x)-\dfrac{X}{2}f''(x)]\\ = \lim\limits_{n\rightarrow\infty}n^2\left[(\mathcal{B}_n^{\alpha,\beta}(f(t);x)-f(x)-
T_{n,1}^{\alpha,\beta}(x)f'(x)-T_{n,2}^{\alpha,\beta}(x)\dfrac{f''(x)}{2}\right]+\\
+\lim\limits_{n\rightarrow\infty}n[nT_{n,1}^{\alpha,\beta}(x)-(\alpha+1-(\alpha+\beta+2)x)]f'(x)+\\
+\dfrac{1}{2}\lim\limits_{n\rightarrow\infty}n[nT_{n,2}^{\alpha,\beta}(x)-X]f''(x).
\end{array}
$$
By using (\ref{beta_eq1.0}),  (\ref{beta_eq2.1}) and  (\ref{beta_eq1.48}) with $l=2$, we get
$$
\begin{array}{l}
\lim\limits_{n\rightarrow\infty}n[n(\mathcal{B}_n^{\alpha,\beta}(f(t);x)-f(x))-\dfrac{X}{2}f''(x)-(\alpha+1-(\alpha+\beta+2)x)f'(x)]\\
=\dfrac{1}{8}X^2f^{IV}(x)+\dfrac{1}{6}X(3\alpha+5-(3\alpha+3\beta+10)x)f'''(x)-\\-(\alpha+\beta+2)(\alpha+1-(\alpha+\beta+2)x)f'(x)+\dfrac{1}{2}f''(x)[(\alpha+1)(\alpha+2)-\\-(2\alpha^2+2\alpha\beta+10\alpha+4\beta+11)x+x^2((\alpha+\beta)(\alpha+\beta+7)+11)].
\end{array}
$$
For $\alpha=\beta=-1$, this reduces to
$$
\begin{array}{l}
\lim\limits_{n\rightarrow\infty}n[n(\mathcal{B}_n^{-1,-1}(f;x)-f(x))-\dfrac{X}{2}f''(x)]=\\=\dfrac{1}{24}\left(3X^2f^{IV}(x)+8X(1-2x)f'''(x)-12Xf''(x)\right).
\end{array}
$$
This result can be also deduced from \cite[Remark 3]{Abel-Gupta-Mohapatra:2005}.
\end{remark}
%===========================================
\section{Iterates of $\mathcal{B}_n^{\alpha,\beta}$}
%===========================================
1. $\alpha=\beta=-1$. In this case $\mathcal{B}_n^{-1,-1}$ are positive linear operators preserving  linear functions, and $\mathcal{B}_n^{-1,-1}e_2(x)=\dfrac{nx(nx+1)}{n(n+1)}>e_2(x)$, for $0<x<1$. Consequently
$$
\lim\limits_{m\rightarrow\infty}\left(\mathcal{B}_n^{-1,-1}\right)^mf(x)=(1-x)f(0)+xf(1), f\in C[0,1],
$$
uniformly on $[0,1]$ ( see \cite{Rasa:2010}).\vspace{0.2cm}\\
2. $\alpha>-1, \beta=-1$. Then $\mathcal{B}_n^{\alpha,-1}$ are positive linear operators preserving constant functions, $\mathcal{B}_n^{\alpha,-1}f(1)=f(1)$ for all $f\in C[0,1]$, and
$$
\mathcal{B}_n^{\alpha,-1}e_2(x)=\dfrac{(nx+\alpha+1)(nx+\alpha+2)}{(n+\alpha+1)(n+\alpha+2)}>e_2(x), 0\leq x< 1.
$$
Therefore
$$
\lim\limits_{m\rightarrow\infty}\left(\mathcal{B}_n^{\alpha,-1}\right)^mf(x)=f(1), f\in C[0,1],
$$
uniformly on $[0,1]$ (see \cite{Rasa:2010}).\vspace{0.2cm}\\
3. $\alpha=-1, \beta>-1$. As in the previous case, one proves that
$$
\lim\limits_{m\rightarrow\infty}\left(\mathcal{B}_n^{-1,\beta}\right)^mf(x)=f(0), f\in C[0,1].
$$
4. $\alpha>-1, \beta>-1$. In this case we have for all $k\geq 0$,
$$
\mathcal{B}_n^{\alpha,\beta}e_k(x)=\dfrac{(nx+\alpha+1)^{\overline{k}}}{(n+\alpha+\beta+2)^{\overline{k}}}, x\in [0,1].
$$
From this we get
\begin{equation}\label{beta_eq3.1}
\mathcal{B}_n^{\alpha,\beta}e_k(x)=\dfrac{1}{(n+\alpha+\beta+2)^{\overline{k}}}\sum\limits_{j=0}^{k}s_{k-j}(k,\alpha)n^jx^j,
\end{equation}
where $s_{k-j}(k,\alpha)$ are elementary symmetric sums of the numbers $\alpha+1, \alpha+2,..., \alpha+k$; in particular
$s_{0}(k,\alpha)=1$ and
\begin{equation}\label{beta_eq3.2}
s_{1}(k,\alpha)=(\alpha+1)+...+(\alpha+k)=k\alpha+\dfrac{k(k+1)}{2}.
\end{equation}
It follows that the numbers
\begin{equation*}
\lambda_{n,k}:=\dfrac{n^k}{(n+\alpha+\beta+2)^{\overline{k}}}, k\geq 0,
\end{equation*}
are eigenvalues of $\mathcal{B}_n^{\alpha,\beta}$, and to each of them there corresponds a monic eigenpolynomial $p_{n,k}$ with $deg\; p_{n,k}=k$. Let $p\in \Pi$ and $d=deg\; p$. Then $p$ has a decomposition
$$
p=a_{n,0}(p)p_{n,0}+a_{n,1}(p)p_{n,1}+...+a_{n,d}(p)p_{n,d}
$$
with some coefficients $a_{n,k}(p)\in \mathbb{R}$. Since $\lambda_{n,0}=1$ and $p_{n,0}=e_0$ we get
$$
(\mathcal{B}_n^{\alpha,\beta})^mp=a_{n,0}(p)e_0+\sum\limits_{k=1}^d a_{n,k}(p)\lambda_{n,k}^m p_{n,k}, \; m\geq 1
$$
and so
\begin{equation}\label{beta_eq3.3}
\lim\limits_{m\rightarrow\infty}(\mathcal{B}_n^{\alpha,\beta})^mp=a_{n,0}(p)e_0, \;p\in \Pi.
\end{equation}
Consider the linear functional $\mu_n:\Pi\rightarrow\mathbb{R}, \mu_n(p)=a_{n,0}(p)$, and the linear operator $P_n:\Pi\rightarrow\Pi$,
\begin{equation*}
P_np=\mu_n(p)e_0, \;p\in \Pi.
\end{equation*}
Then (\ref{beta_eq3.3}) becomes
\begin{equation}\label{beta_eq3.4}
\lim\limits_{m\rightarrow\infty}(\mathcal{B}_n^{\alpha,\beta})^mp=P_np, \;p\in \Pi.
\end{equation}
Obviously $P_n$ is positive, and so $\mu_n$ is positive; moreover, $||\mu_n||=1$ because $\mu_n(e_0)=1$. By the Hahn-Banach theorem, $\mu_n$ can be extended to a norm-one linear functional on $C[0,1]$. Since $\Pi$ is dense in $C[0,1]$, the extension is unique and the extended functional $\mu_n:C[0,1]\rightarrow\mathbb{R}$ is also positive. Now $P_n$ can be extended from $\Pi$ to $C[0,1]$ by setting $P_n:C[0,1]\rightarrow\Pi, P_nf=\mu_n(f)e_0, f\in C[0,1]$.
Remark that
\begin{equation}\label{beta_eq3.5}
||(\mathcal{B}_n^{\alpha,\beta})^m||=||P_n||=1, \;m\geq 1.
\end{equation}
Using again the fact that $\Pi$ is dense in $C[0,1]$, we get from (\ref{beta_eq3.4}) and (\ref{beta_eq3.5})
\begin{equation}\label{beta_eq3.6}
\lim\limits_{m\rightarrow\infty}(\mathcal{B}_n^{\alpha,\beta})^mf=P_nf, \;f\in C[0,1].
\end{equation}
On the other hand, from (\ref{beta_eq3.1}) we deduce the following recurrence formula for the computation of $P_n e_k, k\geq 1$:
\begin{equation*}
\left((n+\alpha+\beta+2)^{\overline{k}}-n^k\right)P_ne_k=\sum\limits_{j=0}^{k-1}s_{k-j}(k,\alpha)n^jP_ne_j.
\end{equation*}
Since $P_ne_k=\mu_n(e_k)e_0$, we get for $n\geq 1$ and $k\geq 1$
\begin{equation}\label{beta_eq3.7}
\mu_n(e_k)=\sum\limits_{j=0}^{k-1}s_{k-j}(k,\alpha)\dfrac{n^j}{(n+\alpha+\beta+2)^{\overline{k}}-n^k}\mu_n(e_j).
\end{equation}
Using (\ref{beta_eq3.7}) it is easy to prove by induction on $k$ that there exists
\begin{equation}\label{beta_eq3.8}
\mu(e_k):=\lim\limits_{n\rightarrow\infty}\mu_n(e_k), k\geq 0,
\end{equation}
and, moreover,
$$
\mu(e_k)=\dfrac{s_1(k,\alpha)}{(\alpha+\beta+2)+...+(\alpha+\beta+k+1)}\mu(e_{k-1}),
$$
i.e., taking (\ref{beta_eq3.2}) into account,
$$
\mu(e_k)=\dfrac{2\alpha+k+1}{2\alpha+2\beta+k+3}\mu(e_{k-1}),\; k\geq 1.
$$
Since $\mu(e_0)=1$, it follows that
$$
\mu(e_k)=\dfrac{(2\alpha+2)^{\overline{k}}}{(2\alpha+2\beta+4)^{\overline{k}}},\; k\geq 0.
$$
This can be rewritten as
$$
\mu(e_k)=\dfrac{B(2\alpha+k+2, 2\beta+2)}{B(2\alpha+2, 2\beta+2)}=\dfrac{\int\limits_0^1t^{2\alpha+1}(1-t)^{2\beta+1}e_k(t)dt}{\int\limits_0^1t^{2\alpha+1}(1-t)^{2\beta+1}dt},
$$
so that
$$
\mu(p)=\dfrac{\int\limits_0^1t^{2\alpha+1}(1-t)^{2\beta+1}p(t)dt}{\int\limits_0^1t^{2\alpha+1}(1-t)^{2\beta+1}dt}, p\in\Pi.
$$
Consider the extension of $\mu$ to $C[0,1]$, i.e.,
$$
\mu(f)=\dfrac{\int\limits_0^1t^{2\alpha+1}(1-t)^{2\beta+1}f(t)dt}{\int\limits_0^1t^{2\alpha+1}(1-t)^{2\beta+1}dt}, f\in C[0,1],
$$
and the positive linear operator $P:C[0,1]\rightarrow\Pi, Pf=\mu(f)e_0, f\in C[0,1]$. Acording to (\ref{beta_eq3.8}), $\lim\limits_{n\rightarrow\infty}\mu_n(p)=\mu(p), p\in\Pi$, i.e.,
\begin{equation}\label{beta_eq3.9}
\lim\limits_{n\rightarrow\infty}P_np=Pp, p\in\Pi.
\end{equation}
Since $||P_n||=||P||=1, n\geq 1$, we conclude from (\ref{beta_eq3.9}) that $\lim\limits_{n\rightarrow\infty}P_nf=Pf, f\in C[0,1]$. Thus, for the operators $P_n$ described in (\ref{beta_eq3.6}) we have proved:
\begin{theorem}
Let $\alpha>-1, \beta>-1$. Then for each $f\in C[0,1]$ and $n\geq 1$,
\begin{equation*}
\lim\limits_{n\rightarrow\infty}P_nf=\dfrac{\int\limits_0^1t^{2\alpha+1}(1-t)^{2\beta+1}f(t)dt}{\int\limits_0^1t^{2\alpha+1}(1-t)^{2\beta+1}dt}e_0.
\end{equation*}
\end{theorem}
For $\alpha=\beta=0$, this result was obtained, with different methods, in \cite{Attalienti-Rasa:2008}.

\bigskip
 \noindent
$\begin{array}{ll}
\textrm{Heiner Gonska}\\
 \textrm{University of Duisburg-Essen} \\
 \textrm{Faculty for Mathematics} \\
 \textrm{47048 Duisburg, Germany} \\
\mathtt{heiner.gonska@uni-due.de} 
\end{array} \qquad \qquad $
 $\begin{array}{ll}
\textrm{Ioan Ra\c sa}\\
 \textrm{Technical University of Cluj-Napoca}\\
 \textrm{Department of Mathematics} \\
 \textrm{400114 Cluj-Napoca, Romania} \\
\mathtt{Ioan.Rasa@math.utcluj.ro} 
\end{array} \qquad $
\; \vspace{0.1cm}
$\begin{array}{l}
 \vspace{0.5cm}\\
\textrm{Elena-Dorina St\u anil\u a}\\
 \textrm{University of Duisburg-Essen} \\
 \textrm{Faculty for Mathematics} \\
 \textrm{47048 Duisburg, Germany} \\
\mathtt{elena.stanila@stud.uni-due.de} 
\end{array} $


\begin{thebibliography}{99}

\bibitem{Abel-Gupta-Mohapatra:2005} U. Abel, V. Gupta, R.N. Mohapatra: Local approximation by Beta operators, {\it Nonlinear Analysis} {\bf 62} (2005), 41-52.
\bibitem{Attalienti-Rasa:2008} A. Attalienti, I. Ra\c sa: Overiterated linear operators and asymptotic
behavior of semigroups, {\it Mediterr. J. Math.} {\bf 5} (2008), 315-324.
\bibitem{Lupas:1972} A. Lupa\c{s}: {\it Die Folge der Betaoperatoren}, Ph.D Thesis, Stuttgart: Universit\"at Stuttgart, 1972.
\bibitem{Mamedov:1962} R.G. Mamedov: The asymptotic value of the approximation of multiply differentiable functions by positive linear operators. (Russian) {\it Dokl. Akad. Nauk SSSR} {\bf 146} (1962), 1013-1016.  
\bibitem{Muhlbach:1972} G. M\"uhlbach: Rekursionsformeln f\"ur die zentralen Momente der P\'olya und der Beta-Verteilung, \textit{Metrika} \textbf{19} (1972), 171-177.
 \bibitem{Rasa:2010} I. Ra\c sa: $C_0$ - semigroups and iterates of positive linear operators:
asymptotic behaviour, {\it Rendiconti del Circolo Matematico di Palermo}, Ser.
II, Suppl. {\bf 82} (2010), 123-142.
\bibitem{Sikkema:1975} P.C. Sikkema: \"Uber die Schurerschen linearen positiven Operatoren. I. (German) {\it Nederl. Akad. Wetensch. Proc. Ser. A 78 = Indag. Math.} {\bf 37} (1975), 230-242.
%
\end{thebibliography}
\end{document}